\documentclass[reqno,12pt]{amsart}
\def\vint_#1{\mathchoice%
          {\mathop{\kern 0.2em\vrule width 0.6em height 0.69678ex depth -0.58065ex
                  \kern -0.8em \intop}\nolimits_{\kern -0.4em#1}}%
          {\mathop{\kern 0.1em\vrule width 0.5em height 0.69678ex depth -0.60387ex
                  \kern -0.6em \intop}\nolimits_{#1}}%
          {\mathop{\kern 0.1em\vrule width 0.5em height 0.69678ex
              depth -0.60387ex
                  \kern -0.6em \intop}\nolimits_{#1}}%
          {\mathop{\kern 0.1em\vrule width 0.5em height 0.69678ex depth -0.60387ex
                  \kern -0.6em \intop}\nolimits_{#1}}}
\def\vintslides_#1{\mathchoice%
          {\mathop{\kern 0.1em\vrule width 0.5em height 0.697ex depth -0.581ex
                  \kern -0.6em \intop}\nolimits_{\kern -0.4em#1}}%
          {\mathop{\kern 0.1em\vrule width 0.3em height 0.697ex depth -0.604ex
                  \kern -0.4em \intop}\nolimits_{#1}}%
          {\mathop{\kern 0.1em\vrule width 0.3em height 0.697ex de pth -0.604ex
                  \kern -0.4em \intop}\nolimits_{#1}}%
          {\mathop{\kern 0.1em\vrule width 0.3em height 0.697ex depth -0.604ex
                  \kern -0.4em \intop}\nolimits_{#1}}}

\usepackage{a4wide}
\usepackage{mathtools}
\usepackage{amsthm}
\usepackage{amssymb}
\usepackage{hyperref}
\usepackage{color}
\mathtoolsset{showonlyrefs}
\usepackage[obeyFinal]{todonotes}
\usepackage{csquotes}
\usepackage{graphicx}
\usepackage{epsfig,color}
\usepackage{enumitem}
\usepackage{esint}
\numberwithin{equation}{section}
\newtheorem{theorem}{Theorem}[section]
\newtheorem{lemma}[theorem]{Lemma}
\newtheorem{pro}[theorem]{Proposition}
\newtheorem{corollary}[theorem]{Corollary}
\theoremstyle{definition}
\newtheorem{definition}[theorem]{Definition}
\theoremstyle{remark}

\renewcommand{\d}{{\mathrm d}}
\newcommand{\R}{\mathbb{R}}

\newcommand{\N}{\mathbb{N}}

\newcommand{\F}{\mathcal{F}}

\newcommand{\Int}{\mathrm{int\,}}
\newcommand{\diam}{\mathrm{diam}}
\newcommand{\dist}{\mathrm{dist}}

\begin{document}

\title[A density result on Orlicz-Sobolev spaces in the plane]
{A density result on Orlicz-Sobolev spaces in the plane}

\author{Walter A.  Ortiz}
\address{Department of Mathematics\\
            Universitat Aut\`onoma de Barcelona\\
            Spain}
\email{waortiz@mat.uab.cat}

\author{Tapio Rajala}
\address{University of Jyvaskyla \\
         Department of Mathematics and Statistics \\
         P.O. Box 35 (MaD) \\
         FI-40014 University of Jyvaskyla \\
         Finland}
\email{tapio.m.rajala@jyu.fi}

\thanks{TR acknowledges the support by the Academy of Finland, projects no. 274372 and 314789.}
\thanks{WO acknowledges the support by  MINECO Grants MTM201344304P and EEBBI1812876 (Spain).}
\date{\today}


\begin{abstract}
We show the density of smooth Sobolev functions $W^{k,\infty}(\Omega)\cap C^\infty(\Omega)$  in the Orlicz-Sobolev spaces $L^{k,\Psi}(\Omega)$ for bounded simply connected planar domains $\Omega$ and doubling Young functions $\Psi$.
\end{abstract}

\maketitle

\section{Introduction}

Orlicz-Sobolev spaces appear naturally in analysis  as generalizations of the usual Sobolev spaces, for instance when one studies sharp assumptions for mappings of finite distortion \cite{IKO,KKMOZ}.
Orlicz-Sobolev spaces appear also in many other contexts and have been studied by their own right, see for instance \cite{AHS,A75,B10,C96b,C96,DJthesis,dS81,FLS,EGO,H12,H10,HT10,KJF,RR,Tuominen} for a 
sample of the literature.

An important basic question in the theory of function spaces is the relation between different spaces. Answers to this question can be given for instance in terms of embeddings and density results.
In this paper, we show that in Orlicz-Sobolev spaces on bounded simply connected planar domains we can approximate functions with bounded derivatives if we consider only the highest order derivatives in the norm.
\begin{theorem}\label{thm:main}
Let $k\in\mathbb{N}$, $\Psi$ be a doubling Young function, and $\Omega\subset\mathbb{R}^2$ be a bounded simply connected domain. Then the subspace $W^{k,\infty}(\Omega)\cap C^{\infty}(\Omega)$ is dense in the space $L^{k,\Psi}(\Omega).$
\end{theorem}

Recall that for a domain $\Omega \subset\mathbb{R}^2$ and a Young function $\Psi$, the version of the Orlicz-Sobolev space $L^{k,\Psi}(\Omega)$ used in Theorem \ref{thm:main} is defined as
\begin{align*}
L^{k,\Psi}(\Omega)=\left\{f\in L^{\Psi}(\Omega):\nabla^{\alpha}f\in L^{\Psi}(\Omega)\;\text{ if } |\alpha|=k\right\}.
\end{align*}
The space $L^{k,\Psi}(\Omega)$ is equipped
with the semi-norm $\sum_{|\alpha|=k}\left\|\nabla^{\alpha}f\right\|_{L^{\Psi}(\Omega)}$, where \\
$\|\cdot\|_{L^{\Psi}(\Omega)}$ is the Luxemburg norm. See Section \ref{sec:pr} for more basic information on Orlicz and Orlicz-Sobolev spaces.

We will use a Whitney decomposition of the domain $\Omega \subset \mathbb{R}^2$  and make a polynomial approximation near the boundary $\partial\Omega$. The validity of the approximation is proven using a $\Psi-\Psi$ Poincar\'e inequality.
The form of the polynomial approximation we use here was introduced in \cite{DRS}, where a density result was shown for homogeneous Sobolev spaces on simply connected planar domains. This was then extended to Gromov hyperbolic domains in higher dimensions in \cite{Nandi}.
In turn, both of these results were generalizations of density results for first order Sobolev spaces \cite{KZ,KTZ} which were partly motivated by the recent progress on planar Sobolev extension domains \cite{KRZ,Shvartsman}.

Although for every domain smooth functions are dense in $W^{k,p}(\Omega)$ \cite{MS} and, more generally, in $W^{k,\Psi}(\Omega)$ for doubling $\Psi$ \cite{DT}, the derivatives of the approximating smooth functions might blow up near the boundary. Therefore, the density of other function spaces, such as $W^{k,q}(\Omega)$ in $W^{k,p}(\Omega)$ might be false, see \cite{Koskela,KTZ} for this, and for instance \cite{Amick,Kolsrud,OFarrell} for earlier counter examples on other function spaces.
The density of global smooth functions in $W^{1,p}(\Omega)$ is known for instance for Jordan domains \cite{Lewis,KZ}, but the case for higher order Sobolev spaces is still open. Similarly, in the case $k \ge 2$, the density result presented in Theorem \ref{thm:main} remains still open for the full Orlicz-Sobolev space $W^{k,\Psi}(\Omega)$ as well as for the usual Sobolev space $W^{k,p}(\Omega)$.

In the same way as for the first order Sobolev spaces \cite{KZ}, for $W^{1,\Psi}(\Omega)$ we get a better density result as a corollary of Theorem \ref{thm:main}. This is simply because we may first cut a function in $W^{1,\Psi}(\Omega)$ from above and below introducing a small error in the norm, so that the function becomes an $L^\infty(\Omega)$ function. The remaining approximations do not change the fact that the function is in $L^\infty(\Omega)$.
\begin{corollary}
Let $\Psi$ be a doubling Young function and $\Omega\subset\mathbb{R}^2$ be a bounded simply connected domain. Then the subspace $W^{1,\infty}(\Omega)\cap C^{\infty}(\Omega)$ is dense in the space $W^{1,\Psi}(\Omega).$
\end{corollary}

The paper is organized as follows. In Section \ref{sec:pr} we recall the Whitney decomposition and list the required prerequisites from the Orlicz theory. In Section \ref{sec:pu} we give a partition of unity for the domain using a Whitney type decomposition. Finally, in Section \ref{sec:thm} we show the proof of Theorem \ref{thm:main}.

\section{Preliminaries}\label{sec:pr}

In this paper, we will usually denote constants by $C$. The value of the constant might change between appearances, even in a chain of inequalities, but the dependence of the constant on a set of fixed parameters is always stated. Sometimes, to clarify the dependence, the parameters are written inside parentheses $C(\cdot)$. 

\subsection{Whitney decomposition}

In this section we recall the Whitney decomposition of a domain in $\R^d$.
Such decomposition is standard in analysis, see for instance Whitney \cite{whitney} or the book of Stein \cite[Chapter VI]{stein}.
We will use a version of the decomposition that was used in \cite{DRS}.

We denote the sidelength of a square $Q$ by $\ell(Q)$.
For notational convenience we start the Whitney decomposition below from squares with sidelength $2^{-1}$. Formally, since we are working with doubling Young functions, by rescaling, we may consider all bounded domains $\Omega \subset \R^2$ to have $\diam(\Omega) \le 1$ in which case no Whitney decomposition would have squares larger than the ones used below regardless of the starting scale.
A Whitney decomposition in the plane consists of dyadic squares. Let us first recall those.

\begin{definition}[Dyadic squares]
A dyadic interval in $\mathbb{R}$ is an interval of the form\\ $[m2^{-k},(m+1)2^{-k}]$ where $m,k\in\mathbb{Z}.$ A dyadic square in $\mathbb{R}^2$ is a product of dyadic intervals of the
same length. That is, a dyadic square is a set of the form
\begin{align*}
[m_12^{-k},(m_1+1)2^{-k}] \times [m_22^{-k},(m_2+1)2^{-k}]
\end{align*}
for some integers $m_1$ and $,m_2$.
\end{definition}

Let us now define a Whitney decomposition following Lemma 2.3 in \cite{DRS}.

\begin{definition}[Whitney decomposition]\label{def:wid}
Let $\Omega$ be a bounded open subset of $\R^2$. A Whitney decomposition is a collection $\tilde\F$ of dyadic squares inside $\Omega$ satisfying the following properties.
\begin{enumerate}[label=\textbf({W\arabic*)},ref=(W\arabic*)]
\item $\Omega=\bigcup_{Q\in \tilde\F}Q $ \label{eka}
\item $\ell(Q)<\dist(Q,\Omega^c)\le3\sqrt2 \ell(Q)=3\diam(Q)$ for all $Q\in\tilde\F$\label{toka}
\item $\Int Q_1\cap \Int Q_2= \emptyset$ for all $Q_1,Q_2\in \tilde\F, Q_1\ne Q_2$\label{kolmas}
\item If $Q_1,Q_2\in \tilde\F$ and $Q_1\cap Q_2\neq\emptyset$, then $\frac{\ell(Q_1)}{\ell(Q_2)}\leq 2$.\label{neljas}
\end{enumerate}
\end{definition}
Suppose $Q_1,\dots,Q_m$ are Whitney squares such that $Q_j$ and $Q_{j+1}$ touch and $\frac{1}{4}\leq \frac{\ell(Q_j)}{\ell(Q_{j+1})}\leq 4$ for all $j,1\leq j\leq m-1.$ We say then $\{Q_1,\dots,Q_m\}$ is a \textit{chain} connecting   $Q_1$ to $Q_m$ and define the length of that chain to be the integer $m$.\\

\subsection{Orlicz Spaces}
\begin{definition}
A function $\Psi\colon [0,\infty)\rightarrow[0,\infty]$ is a \textit{Young function}
if
\begin{align}
\Psi(s)=\displaystyle\int_{0}^s\psi(t)\,\d t,
\end{align}
where  $\psi:[0,\infty]\rightarrow[0,\infty]$ is an increasing, left continuous function which is neither identically zero nor identically infinite on $(0,\infty)$ and which satisfies $\psi(0)=0.$
\end{definition}

A Young function $\Psi$ is convex, increasing, left continuous, $\Psi(0)=0$ and $\Psi(t)\rightarrow \infty $ as $t\rightarrow \infty.$ A continuous Young function with the properties $\Psi(t)=0,$ only if $t=0,$ $\Psi(t)/t\rightarrow \infty $ as $t\rightarrow \infty $ and $\Psi(t)/t\rightarrow 0$ as $t\rightarrow 0$ is called an $N$-function.\\


It follows easily from the convexity and $\Psi(0)=0$, that the function $t\rightarrow\Psi(t)/t$ is increasing. This implies that if $\Psi$ is strictly increasing, then the function $\Psi^{-1}(t)/t$ is decreasing.
A Young function $\Psi$ is doubling if there is a constant $C>0$ such that 
\begin{equation}\label{eq:doubling}
\Psi(2t)\leq C\Psi(t)
\end{equation}
for each $t\geq0$. The smallest constant $C$ satisfying \eqref{eq:doubling} is called the doubling constant of $\Psi$.
\begin{definition}
Given a doubling Young function $\Psi$ and an open set $\Omega\subset \R^2$,
we denote by $L^{\Psi}(\Omega)$, the \textit{Orlicz space} associated to $\Psi$, defined by
 $$
 L^{\Psi}(\Omega)=\left\{ u\colon \Omega\rightarrow\mathbb{R}:\displaystyle\int_{\Omega}\Psi(|u(x)|)\,\d x<\infty\right\}.
 $$
 $L^{\Psi}(\Omega)$ is a Banach space, when equipped with the \textit{Luxemburg norm}
 \[
 \|u\|_{L^\Psi(\Omega)} = \inf\left\{k > 0\,:\, \displaystyle\int_{\Omega}\Psi(k^{-1}|u(x)|)\,\d x\le 1\right \}.
 \]
 \end{definition}

We will not use the Luxemburg norm in this paper, but work with the integrals. This is justified by the following fact.
\begin{lemma}\label{lma:modnorm}
Let $\Psi$ be a doubling Young function. Then
\[
\|u_i-u\|_{L^\Psi(\Omega)} \to 0 \qquad\text{as }i \to \infty
\]
if and only if
\[
\displaystyle\int_{\Omega}\Psi(|u_i(x)-u(x)|)\,\d x \to 0 \qquad\text{as }i \to \infty.
\]
\end{lemma}

A direct consequence of Jensen's inequality is the following.

 \begin{lemma}\label{lma:jensen}
Let $\Psi$ be a Young function, $u\in L_{\text{loc}}^1(\R^2)$ and $A\subset \R^2$ of positive and finite measure, then 
$$
\Psi\left(\fint_{A}|u(x)|\,\d x\right)\leq\fint_{A}\Psi(|u(x)|)\,\d x,
$$
where $u_A=\fint_{A}|u(x)|\,\d x$ is the average integral.
 \end{lemma}
 
\subsection{Poincar\'e inequalities and polynomial approximation}
From now on, $\Psi$ always refers to a doubling Young function.
We will construct an approximation by replacing the original function by approximating polynomials near the boundary of the domain. For this purpose, we will need a few lemmas regarding the polynomials. Here and later on by $|E|$ we denote the Lebesgue measure of a set $E \subset \R^2$.

\begin{lemma}\label{norm_equivalence}
Let $Q$ be any square in $\mathbb{R}^2$ and $P$ be a polynomial of degree $k$ defined in $\mathbb{R}^2$. Let $E,F\subset Q$ be such that $|E|,|F|>\eta |Q|$ where $\eta>0$. Then
$$\int_E \Psi(|P(x)|)\,\d x\leq C \int_F \Psi(|P(x)|)\,\d x,$$
where the constant $C$ depends only on $\eta$, $k$ and the doubling constant of $\Psi$.
\end{lemma}
\begin{proof}
Since any two norms on a finite dimensional vector space are comparable, there exists a constant $\delta>0$ depending only on $k$ and $\eta$ so that
the set
$$\tilde{F}:=\left\{x\in F:|P(x)|\geq\delta\max_{y\in Q}|P(y)| \right\}$$ 
satisfies $|\tilde{F}|\geq \frac12|F|$.

Therefore, by applying monotonicity and doubling properties on $\Psi$ we obtain
\begin{align*}
\displaystyle\int_F \Psi (|P(x)|)\,\d x&\geq \displaystyle\int_{\tilde{F} }\Psi (|P(x)|)\,\d x
\geq\displaystyle\int_{\tilde{F} }\Psi ( \delta\max_{y\in Q}|P(y)|)\,\d x\\
&\geq C\displaystyle\int_{\tilde{F} }\Psi (\max_{y\in Q}|P(y)|)\,\d x
=C|\tilde{F}|\Psi (\max_{y\in Q}|P(y)|)\\
&\geq \frac{C|\tilde{F}|}{|Q|}\displaystyle\int_Q \Psi (|P(x)|)\,\d x
 \geq C\displaystyle\int_Q \Psi (|P(x)|)\,\d x\\
&\geq C\displaystyle\int_E\Psi ( P(x))\, \d x,
\end{align*}
which gives the claim.
\end{proof}
Given a function $u\in C^{\infty}(\Omega)$, degree $k \in \N$, and a bounded set $E\subset\Omega$ with $|E|>0$, we define (see \cite{J}) the polynomial approximation of $u$ in $E$, $P_k(u,E)$ to be the polynomial of order $k-1$ which satisfies
$$\int_E \nabla^{\alpha}(u-P_k(u,E))=0$$ for each $\alpha=(\alpha_1,\alpha_2)$ such that $|\alpha|=\alpha_1+\alpha_2\leq k-1$. Once $k$ is fixed, we denote the polynomial approximation of $u$ in a dyadic square $Q$ as $P_Q$.

\begin{pro}[$\Psi-\Psi$ Poincar\'e inequality]\label{prop:poinc}
Let $k,m \in \N$. There exists a constant $C$ depending only on $k$, $m$ and the doubling constant of $\Psi$ such that for any domain $\Omega\subset \mathbb{R}^2$, a chain $\{Q_i\}_{i=1}^m$ of dyadic squares in  $\Omega$, and a function $u\in L^{k,\Psi}(\Omega)$ we have
\begin{align*}
\fint_{E}\Psi\left(\frac{|u(x)-P_E(x)|}{\ell(Q_1)^k}\right)\,\d x\le C\fint_{E}\Psi(|\nabla^k u(x)|)\,\d x,
\end{align*} 
where we have abbreviated $E = \bigcup_{i=1}^m Q_i$.
\end{pro}
\begin{proof}
By \cite[Lemma 1]{BL1991} the claim is true for $k = 1$ in the case where $E$ is convex. By a change of variables, the claim extends to the case $k=1$ and $E = \bigcup_{i=1}^m Q_i$, for the chain $\{Q_i\}_{i=1}^m$. What remains to show is the case $k > 1$.

We do this by induction. Suppose the claim is true for the order $k-1$. Then, using the Poincar\'e inequality first for the $k-1$ orders and the for the first order, we obtain
\begin{align*}
\fint_E\Psi\left(\frac{|u(x)-P_E(x)|}{\ell(Q_1)^{k}}\right)\,\d x 
& \le C \fint_E\Psi\left(\frac{|\nabla^{k-1}(u(x)-P_E(x))|}{\ell(Q_1)}\right)\,\d x \\
& \le C \fint_E\Psi\left(|\nabla^{k}(u(x)-P_E(x))|\right)\,\d x\\
& = C \fint_E\Psi\left(|\nabla^{k}u(x)|\right)\,\d x. \qedhere
\end{align*}
\end{proof}

\begin{lemma}\label{le:pi}
 Let $\Omega \subset\mathbb{R}^2$ be a bounded simply connected domain and $\tilde{\mathcal{F}}$ a Whitney decomposition of $\Omega.$ Let $\{Q_i\}_{i=1}^m$ be a chain of dyadic squares in $\tilde{\mathcal{F}}$.
 Then there exists a constant $C$ depending only on $k$, $m$ and the doubling constant of $\Psi$ such that if $ |\alpha|\leq k,$  we have
 \[
  \int_{Q_1}\Psi\left(\frac{\nabla^\alpha P_{Q_1}-\nabla^\alpha P_{Q_m}}{\ell(Q_1)^{k-|\alpha|}}\right)\leq C\int_{\bigcup_{i=1}^mQ_i}\Psi(|\nabla^ku|).
 \]
 \end{lemma}
 \begin{proof}
Let us abbreviate $E = \bigcup_{i=1}^mQ_i$.
Now, using the triangle inequality, Lemma \ref{norm_equivalence}, the doubling property of $\Psi$, then triangle inequality again and Lemma \ref{prop:poinc} we obtain
\begin{align*}
\begin{split}
\displaystyle\int_{Q_1}\Psi & \left(\displaystyle\dfrac{|\nabla^{\alpha}P_{Q_1}-\nabla^{\alpha}P_{Q_{m}}|}{\ell(Q_1)^{k-|\alpha|}}\right)\\
&\le \displaystyle\int_{Q_1}\Psi\left(\displaystyle\dfrac{|\nabla^{\alpha}P_{Q_1}-\nabla^{\alpha}P_{Q_{m}}+\nabla^{\alpha}P_{E}-\nabla^{\alpha}P_{E}|}{\ell(Q_1)^{k-|\alpha|}}\right)\\
&\le C\displaystyle\int_{Q_1}\Psi\left(\displaystyle\dfrac{|\nabla^{\alpha}P_{Q_1}-\nabla^{\alpha}P_{E} |}{\ell(Q_1)^{k-|\alpha|}}\right) + 
C\displaystyle\int_{Q_{m}}\Psi\left(\displaystyle\dfrac{
\nabla^{\alpha}P_{Q_{m}}-\nabla^{\alpha}P_{E}|}{\ell(Q_m)^{k-|\alpha|}}\right)\\
&\leq C\displaystyle\int_{Q_1}\Psi\left(\frac{|\nabla^{\alpha}(u-P_{Q_1})|}{\ell(Q_1)^{k-|\alpha|}}\right)
+ C\displaystyle\int_{Q_{m}}\Psi\left(\frac{|\nabla^{\alpha}(u-P_{Q_{m}})|}{\ell(Q_m)^{k-|\alpha|}}\right)\\
& \quad + C\displaystyle\int_{E}\Psi\left(\frac{|\nabla^{\alpha}(u-P_{E})|}{\ell(Q_1)^{k-|\alpha|}}\right)\\ 
   &\leq C\displaystyle\int_{\bigcup_{i=1}^m Q_i}\Psi(|\nabla^ku|).\qedhere
\end{split}
\end{align*}
\end{proof}

\section{Decomposition and partition of unity}\label{sec:pu}

In this section we recall the decomposition of the domain $\Omega$ and the associated partition of unity that was obtained in \cite{DRS}. In order to make the comparison between this paper and \cite{DRS} easy, we use here the notation from \cite{DRS}.

\subsection{Decomposition of the domain}

We fix a square $Q_0$, which is one of the largest Whitney squares in $\Omega$.
For each $n \in \N$, the domain $\Omega$ is then divided into a core part $D_n$, and a boundary layer, which is the union of sets $\tilde H_i$, see Figure \ref{fig:decomp}.
\begin{figure}
\includegraphics[width=0.8\textwidth]{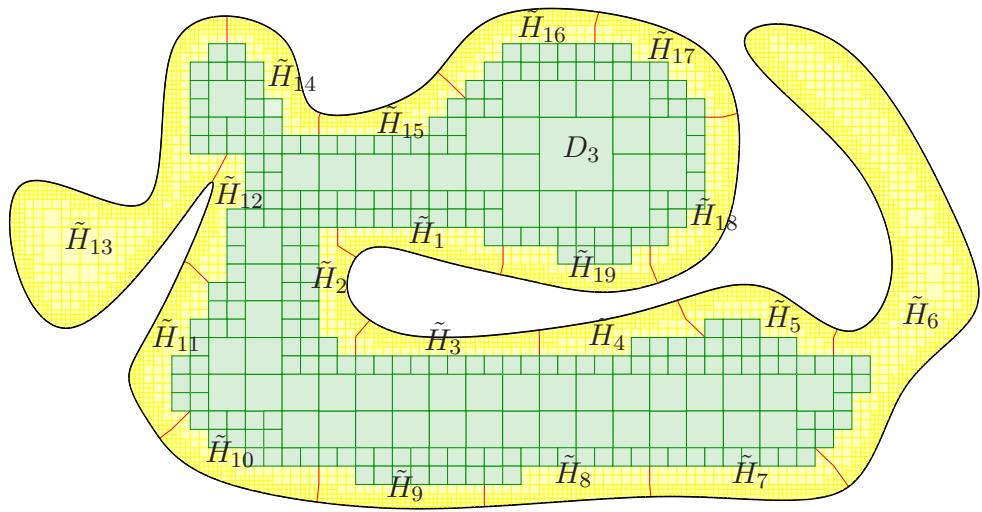}
     \caption{The domain $\Omega$ is decomposed into a core part $D_3$ (obtained as a connected component of a union of Whitney squares) and boundary parts $\tilde H_i$. A partition of unity is made using this decomposition.}
   \label{fig:decomp}
 \end{figure}
The core part $D_n$ is the connected component containing $Q_0$ of the interior of the union of Whitney squares of side-length at least $2^{-n}$. The construction of the boundary parts $\{\tilde H_i\}_{i=1}^l$ is more involved. The sets $\tilde H_i$ are labeled so that $\tilde H_i\cap \tilde H_j \ne \emptyset$ if and only if $|i-j|\le 1$ in a cyclical manner.

The sets $\tilde H_i$ are expanded by taking a connected component $H_i$ of a $2^{-n-3}$ neighbourhood of $\tilde H_i$.
The main property of the decomposition is that these expanded sets still satisfy
$H_i\cap H_j \ne \emptyset$ if and only if $|i-j|\le 1$ in a cyclical manner (Lemma 3.4 in \cite{DRS}). Moreover, since the neighbourhoods are taken in the Euclidean distance, we may use an Euclidean partition of unity.

For each $i$ we associate a Whitney square $Q_i \subset D_n$ of side length $2^{-n}$ so that $H_i \cap Q_i \ne \emptyset$ and, more importantly, if $|i-j|\le 1$, there is a chain of Whitney squares $\mathcal Q_{i,j}$ with length bounded by a universal constant connecting $Q_i$ and $Q_j$.

\subsection{Partition of unity}\label{ssec:pu}

Using the decomposition of the domain $\Omega$ introduced above, we make a partition of unity for the domain. The partition of unity consists of functions $\psi_i$, $i \in \{0,\dots,l\}$ with the following properties:
\begin{enumerate}
\item The function $\psi_0$ is supported in $B(D_{n},\frac{2^{-n}}{10})$.
\item For $i \ge 1$ the function $\psi_i$ is supported in $H_i$.
\item For all $i$, $0\leq \psi_i\leq 1$.
\item $\sum \psi_i \equiv 1$ on $\Omega$.
\item For all $i$, $|\nabla^{\alpha}\psi_i|\leq C({\alpha})2^{-n|\alpha|}$ for all multi-indeces $\alpha$.
\end{enumerate}

\section{Proof of Theorem \ref{thm:main}}\label{sec:thm}

In this section we prove Theorem \ref{thm:main}, using the results of Section \ref{sec:pr} and the partition of union from \cite{DRS} that was recalled in Section \ref{sec:pu}. The polynomial approximation is exactly the same as in \cite{DRS}. What is different is the way the estimates are carried out using Poincar\'e inequalities. Since the usual Poincar\'e inequality is replaced by 
a $\Psi-\Psi$ Poincar\'e inequality (Proposition \ref{prop:poinc}), we need to be more careful with
the chains of inequalities.

Given a function $u \in L^{k,\Psi}(\Omega)$ and $\varepsilon > 0$, our aim is to find a function
$u_{\varepsilon}\in W^{k,\infty}(\Omega)\cap C^{\infty}(\Omega)$ satisfying $\left\|\nabla^ku_{\varepsilon}-\nabla^ku\right\|_{L^{\Psi}(\Omega)}\lesssim \varepsilon.$
We start by noting that we may assume $u \in L^{k,\Psi}(\Omega) \cap C^{\infty}(\Omega)$, since smooth functions are dense in $L^{k,\Psi}(\Omega)$, see \cite{DT}.

For $n \in \N$ fixed, we let $D_n$ and $\{H_i\}_{i=1}^l$ be as in Section \ref{sec:pu}. With these we define
a function $u_n\in W^{k,\infty}(\Omega)\cap C^{\infty}(\Omega)$ by setting for all $x \in \Omega$
\[
u_{n}(x)=\varphi_0(x)u(x)+\sum_{i=1}^{l}\varphi_i(x)P_{i}(x),
\]
where we have abbreviated $P_i := P_{Q_i}$ with the choice of squares $Q_i$ done in Section \ref{ssec:pu}.
Clearly, $u_{n}\in W^{k,\infty}(\Omega)\cap C^{\infty}(\Omega)$.
Therefore, what remains to show is that 
\begin{equation}\label{eq:conv}
\left\|\nabla^ku_{n}-\nabla^ku\right\|_{L^{\Psi}(\Omega)} \to 0, \qquad \text{as }n \to \infty.
\end{equation}

First of all, by the definition of $u_n$, we have 
\[
u_n(x) = u(x) \qquad \text{ for all }x \in \{z \in \Omega\,:\,\varphi_0(z) = 1\} \subset D_{n-1}.
\]
Also, for all $i \in \{1 \dots, l\}$, since $P_i$ is a degree $n-1$ polynomial, we have
\[
\nabla^ku_{n}(x) = 0 \qquad \text{ for all }x \in \{z \in \Omega\,:\,\varphi_i(z) = 1\}.
\]
Therefore,
\begin{align*}
\left\|\nabla^ku_{n}-\nabla^ku\right\|_{L^{\Psi}(\Omega)} & = \left\|\nabla^ku_{n}-\nabla^ku\right\|_{L^{\Psi}(\{\varphi_0 \ne 1\})}\\
& \le \left\|\nabla^ku\right\|_{L^{\Psi}(\Omega \setminus D_{n-1})}
+ \left\|\nabla^ku_{n}\right\|_{L^{\Psi}(\bigcup_{i=1}^l A_i)},
\end{align*}
where we have written $A_i := \{x \in \Omega\,:\,0 < \varphi_i(x) < 1\}$ for $i \in \{1,\dots,l\}$.
Since the sets $D_n$ increasingly exhaust the domain $\Omega$, we have
\[
\left\|\nabla^ku\right\|_{L^{\Psi}(\Omega \setminus D_{n-1})} \to 0 \qquad \text{as }n \to \infty.
\]
Thus, it remains to show that
\begin{equation}\label{eq:remterm1}
\left\|\nabla^ku_{n}\right\|_{L^{\Psi}(\bigcup_{i=1}^lA_i)} \to 0 \qquad \text{as }n \to \infty,
\end{equation}
or equivalently, via Lemma \ref{lma:modnorm}, that for each multi-index $\alpha$ with $|\alpha| = k$ we have
\begin{equation}\label{eq:remterm}
\displaystyle\int_{\bigcup_{i=1}^lA_i}\Psi(|\nabla^{\alpha}u_{n}(x)|) \to 0 \qquad \text{as }n \to \infty.
\end{equation}

In order to show \eqref{eq:remterm}, we estimate for each $i \in \{1,\dots,l\}$ and multi-index $\alpha$ with $|\alpha| = k$, by using the triangle inequality and Jensen's inequality
\begin{align*}
\displaystyle\int_{A_i}\Psi(|\nabla^{\alpha}u_{n}(x)|)\,\d x
& = \displaystyle\int_{A_i}\Psi\left(\left|\nabla^{\alpha}\left(\varphi_0(x)u(x)+\sum_{j=1}^{l}\varphi_j(x)P_{j}(x)\right)\right|\right)\,\d x\\
& \le 
\sum_{\beta\leq\alpha}\displaystyle\int_{A_i}\Psi(|\nabla^{\beta}u(x)-\nabla^{\beta}P_{i}(x)||\nabla^{\alpha-\beta}\varphi_0(x)|)\,\d x\\
&\quad  + \sum_{\beta\leq\alpha}\displaystyle\int_{A_i}\sum_{j=1}^{l}\Psi(|\nabla^{\beta}P_{j}(x)-\nabla^{\beta}P_{i}(x)||\nabla^{\alpha-\beta}\varphi_j(x)|)\,\d x.
\end{align*}
We estimate the above two terms separately.

Let us take $\beta \le \alpha$ and write 
\begin{align*}
\displaystyle\int_{A_i} & \Psi(|\nabla^{\beta}u(x)-\nabla^{\beta}P_{i}(x)||\nabla^{\alpha-\beta}\varphi_0(x)|)\,\d x\\
&  = \sum_{Q}\displaystyle\int_{Q} \Psi(|\nabla^{\beta}u(x)-\nabla^{\beta}P_{i}(x)|C 2^{n(|\alpha|-|\beta|)})\,\d x,
\end{align*}
where the integral is first split into squares $Q$ for which $Q \cap A_i \cap A_0 \ne \emptyset$. There are only a uniformly bounded amount of such squares $Q$, and for each such $Q$ we have
\begin{align*}
\displaystyle\int_{Q} & \Psi(|\nabla^{\beta}u(x)-\nabla^{\beta}P_{i}(x)|C 2^{n(|\alpha|-|\beta|)})\,\d x\\
& \le C\displaystyle\int_{Q} \Psi\left(\frac{|\nabla^{\beta}u(x)-\nabla^{\beta}P_{Q}(x)|}{2^{-n(|\alpha|-|\beta|)}}\right)\,\d x
 + C\displaystyle\int_{Q} \Psi\left(\frac{|\nabla^{\beta}P_Q(x)-\nabla^{\beta}P_{i}(x)|}{2^{-n(|\alpha|-|\beta|)}}\right)\,\d x \\
 & \le C\displaystyle\int_{Q}\Psi(|\nabla^{k}u(x)|)\,\d x + C\displaystyle\int_{A_i}\Psi(|\nabla^{k}u(x)|)\,\d x,
\end{align*}
using the doubling property of $\Psi$, triangle inequality, Proposition \ref{prop:poinc} and Lemma \ref{le:pi}.

Next we estimate for $\beta \le \alpha$,
\begin{align*}
\displaystyle\int_{A_i} \sum_{j=1}^{l} & \Psi(|\nabla^{\beta}P_{j}(x)-\nabla^{\beta}P_{i}(x)||\nabla^{\alpha-\beta}\varphi_j(x)|)\,\d x\\
& = \sum_{j=i-1}^{i+1} \displaystyle\int_{A_i}  \Psi(|\nabla^{\beta}P_{j}(x)-\nabla^{\beta}P_{i}(x)|C 2^{n(|\alpha|-|\beta|)})\,\d x\\
& \le C\sum_{j=i-1}^{i+1} \displaystyle\int_{Q_i}  \Psi\left(\frac{|\nabla^{\beta}P_{j}(x)-\nabla^{\beta}P_{i}(x)|}{2^{-n(|\alpha|-|\beta|)}}\right)\,\d x\\
& \le C\sum_{j=i-1}^{i+1}\displaystyle\int_{A_i\cup A_j}\Psi(|\nabla^{k}u(x)|)\,\d x
\end{align*}
using the fact that $H_i \cap H_j\ne \emptyset$ if and only if $|i-j| \le 1$, Lemma \ref{norm_equivalence} and Lemma \ref{le:pi}.

Combining the above estimates and using the fact that there is only a uniform number of overlaps for the estimates we have
\[
\displaystyle\int_{\bigcup_{i=1}^lA_i}\Psi(|\nabla^{\alpha}u_{n}(x)|) \le 
C \displaystyle\int_{\bigcup_{i=1}^lA_i}\Psi(|\nabla^{\alpha}u(x)|),
\]
giving \eqref{eq:remterm}. This proves the theorem.

\bibliographystyle{amsplain}
\bibliography{SoboReferences}
\end{document}